\documentclass[14pt,cd]{article}
\usepackage{amssymb}
\usepackage{amsfonts}
\usepackage{mathrsfs}
\usepackage{amsmath}
\usepackage{amsthm}
\usepackage{slashed}
\usepackage{url}
\usepackage{bm}
\usepackage{indentfirst}
\usepackage{upgreek}
\usepackage{hyperref}
\usepackage{graphicx}
\usepackage{setspace}
\usepackage[compact]{titlesec}
\usepackage{cite}
\usepackage{abstract}
\usepackage{lmodern}
\usepackage[T1]{fontenc}
\usepackage{xcolor}
\usepackage{float}
\usepackage[all,cmtip]{xy}
\usepackage{mathtools}
\usepackage{fancyhdr}

\def\Eq{\mathop{\rm Eq}}

\newcounter{oftheorem}[section]
\newenvironment{mytheorem}[1]%
{\begin{trivlist}
		\renewcommand{\theoftheorem}{#1~\thesection.\arabic{oftheorem}}
		\refstepcounter{oftheorem}
		\item[\hspace{\labelsep}\bf\thesection.\arabic{oftheorem} #1.]}%
	{\end{trivlist}}
\newenvironment{definition}{\begin{mytheorem}{Definition}}{\end{mytheorem}}
\newenvironment{lemma}{\begin{mytheorem}{Lemma}\it}{\end{mytheorem}}
\newenvironment{proposition}{\begin{mytheorem}{Proposition}\it}{\end{mytheorem}}
\newenvironment{theorem}{\begin{mytheorem}{Theorem}\it}{\end{mytheorem}}
\newenvironment{corollary}{\begin{mytheorem}{Corollary}\it}{\end{mytheorem}}
\newenvironment{remark}{\begin{mytheorem}{Remark}}{\end{mytheorem}}

\begin{document}
	
\title{Projective Covers of 2-star-permutable Categories}	
\author{Vasileios Aravantinos-Sotiropoulos\footnote{The author's research is funded by a doctoral grant Fonds Sp\'eciaux de Recherche of the Universit\'e Catholique de Louvain.}
}

\maketitle	

\begin{abstract}
	We introduce the notion of star-symmetry for relations in a multi-pointed category and use it to obtain a characterization of the projective covers of 2-star-permutable categories. This generalizes the results of Rosick\'y-Vitale for regular Mal'tsev categories \cite{RV}, as well as those of Gran-Rodelo for regular subtractive categories \cite{Jonnson-Tarski}. We apply the characterization in terms of star-symmetry to recover the syntactic conditions defining E-subtractive varieties in the sense of Ursini \cite{U}.
\end{abstract}

\section{Introduction}

The notion of \emph{multi-pointed category} has in recent years been introduced and studied as a setting where certain pointed and non-pointed contexts of interest in Categorical and Universal Algebra can be treated simultaneously. A multi-pointed category is simply a category $\mathcal{C}$ equipped with an \emph{ideal} $\mathcal{N}$ of morphisms in the sense of Ehresmann \cite{Ehresmann}, i.e. a collection of morphisms in $\mathcal{C}$ such that $fg\in\mathcal{N}$ whenever $f\in\mathcal{N}$ or $g\in\mathcal{N}$. The \emph{pointed context} is captured by taking $\mathcal{N}$ to be the class of zero morphisms in a pointed category, while non-pointed settings, which are referred to as the \emph{total context}, are captured by choosing $\mathcal{N}$ to be the class of all morphisms of a category. This has allowed the unification and extension of various results and characterizations known in pointed and non-pointed Categorical Algebra to the context of multi-pointed categories. First, in the article \cite{Good Ideals} the authors introduced the notion of a multi-pointed category with a \emph{good theory of ideals} and unified results from the realm of \emph{ideal determined} categories, on one hand, and \emph{Barr-exact Goursat} categories, on the other. Next, in \cite{Diamonds}, notions of permutability of equivalence relations in multi-pointed categories were introduced and studied in connection with certain diagrammatic characterizations, known for regular \emph{subtractive} categories and \emph{Goursat} categories. Furthermore, in \cite{3 by 3 Lemma} the authors considered generalizations of homological lemmas, such as the \emph{3 $\times$ 3 Lemma} and the \emph{Short Five Lemma}. In non-pointed contexts the appropriate notion of exact sequence is that of \emph{exact fork}, which is a sequence consisting of a kernel pair together with its coequalizer. Then, in a more general multi-pointed context, the pertinent notion becomes that of a \emph{star-exact} sequence, which unifies the pointed and non-pointed versions, and allows for the aforementioned multi-pointed homological lemmas. Finally, in \cite{On 2-star-perm} the notion of 2-star-permutable category was studied as a common extension of both regular subtractive and regular Mal'tsev categories and characterizations of these categories via diagrams such as regular pushouts were generalized to a multi-pointed context. In the present note we want to add to this list a characterization of \emph{projective covers} of regular 2-star-permutable multi-pointed categories.

There has been a lot of interest and work carried out in the literature on obtaining characterizations for the projective covers of various types of regular categories. The first result of this kind appears already in the work of Freyd in \cite{La Jolla} in connection with his construction of the free abelian category on a given (pre-)additive one. About 3 decades later, Carboni and Vitale gave beautiful constructions for free \emph{regular} and \emph{exact} categories. Since abelian categories are in particular exact, these constructions can be used to recover in a nice conceptual manner the aforementioned one by Freyd, as well as other results on abelian categories (see \cite{RV}). One important feature of these \emph{regular and exact completions} is that they apply to any category which is merely \emph{weakly lex}, i.e. which is only required to have weak finite limits \cite{CV}. Then it turns out that any such category $\mathcal{C}$ appears as a projective cover inside both its regular completion and its exact completion and, furthermore, that a free exact category is the exact completion of any one of its projective covers. Such and other motivations have led various authors to establish characterizations for the projective covers of regular and exact categories that are extensive\cite{Gran-Vitale}, Mal'tsev\cite{RV}, protomodular, semi-abelian\cite{Gran}, unital, subtractive\cite{Jonnson-Tarski} and others.

In this note we look at regular Mal'tsev and regular subtractive categories as special cases of the notion of 2-star-permutable category, following the line of research in \cite{Diamonds},\cite{On 2-star-perm}. The aim here is to obtain a characterization of the projective covers of 2-star-permutable multi-pointed categories, thus unifying and subsuming the known characterizations in the Mal'tsev \cite{RV} and subtractive \cite{Jonnson-Tarski} settings. To accomplish this we first prove that 2-star-permutability is equivalent to a certain symmetry property of reflexive relations (\ref{star-symmetry def}, \ref{2-star-permutability characterizations}), which specializes to known characterizations in both the total and pointed contexts. In the total context it becomes the well-known statement \cite{CKP} that a regular category is Mal'tsev if and only if every reflexive relations in it is symmetric, while in the pointed context it says that a regular category is subtractive if and only if every reflexive relation in it is 0-symmetric \cite{AU,Closedness 1}. We then introduce the appropriate ``weakening'' of this symmetry property in the context of multi-pointed categories with only weak finite limits and weak kernels (\ref{star-G-Mal'tsev}) and prove that this weakened property gives the desired characterization of projective covers (\ref{main result}). This result yields, in particular, a characterization of  when the regular completion and the exact completion of a category with weak finite limits are 2-star-permutable. Finally, we apply the result to the case of varieties of universal algebras which have a non-empty set of constants, allowing us to recover the syntactic conditions defining \emph{E-subtractive} varieties in the sense of Ursini\cite{U}.
\vspace{2mm}

\textbf{Acknowledgments:} The author would like to acknowledge with gratitude Professor Marino Gran for numerous helpful conversations and suggestions on the topic and presentation of this paper.
\vspace{3mm}

\section{Preliminaries}

\subsection{Regular categories and relations}

A finitely complete category $\mathcal{E}$ is called \emph{regular} when every kernel pair in $\mathcal{E}$ has a coequalizer and moreover regular epimorphisms in $\mathcal{E}$ are stable under pullbacks. Equivalently, $\mathcal{E}$ is regular if it admits (regular epi, mono) factorizations of morphisms and these are stable under pullback.

A \emph{relation} $R$ from $X$ to $Y$ in any finitely complete category is a subobject $\langle r_{0}, r_{1}\rangle :R\rightarrowtail X\times Y$. When $Y=X$ we will say that $R$ is a relation on $X$ and also denote this by a parallel pair $\xy\xymatrix{R\ar@<1ex>[r]^{r_{0}}\ar@<-1ex>[r]_{r_{1}} & X}\endxy$. The \emph{opposite} relation $R^{\circ}$ is the relation given by $\langle r_{1}, r_{0}\rangle :R\rightarrowtail Y\times X$. Any morphism $f:X\rightarrow Y$ can be considered as a relation by identifying it with its graph $\langle 1_{X}, f\rangle :X\rightarrowtail X\times Y$. Then we will write $f^{\circ}$ to denote the opposite of the latter relation.

In the context of a regular category $\mathcal{E}$ \cite{Barr} it is possible to define a composition of relations which, moreover, is associative. If $R$ is a relation from $X$ to $Y$ and $S$ is a relation from $Y$ to $Z$, then we denote their composition by $SR$, which is a relation from $X$ to $Z$. The diagonal relations $\Delta_{X}=\langle 1_{X}, 1_{X}\rangle:X\rightarrowtail X\times X$ act as identities for the composition of relations on either side. Furthermore, if a relation $R$ is given by the subobject $\langle r_{0}, r_{1}\rangle:R\rightarrowtail X\times Y$, then we can write it as $R=r_{1}r_{0}^{\circ}$ in the above notation.

If $\Eq(f)$ denotes the kernel pair of a morphism $f:X\rightarrow Y$, then as a relation on $X$ we have $\Eq(f)=f^{\circ}f$.

Let $f:X\rightarrow Y$ be a morphism and $S$ be a relation on $Y$. We denote by $f^{-1}(S)$ the \emph{inverse image} of the relation $S$ along $f$, which is the relation on $X$ defined as the pullback of the subobject $S\rightarrowtail Y\times Y$ along the morphism $f\times f:X\times X\rightarrow Y\times Y$. Then in the calculus of relations we have that $f^{-1}(S)=f^{\circ}Sf$.

\subsection{Projective covers}

Let $\mathcal{E}$ be a category with a full subcategory $\mathcal{C}$. We say that $\mathcal{C}$ is a \emph{projective cover} of $\mathcal{E}$ if the following two conditions hold:
\begin{itemize}
	\item Every object of $\mathcal{C}$ is a regular projective in $\mathcal{E}$.
	\item For every object $E\in\mathcal{E}$ there exists a regular epimorphism $P\twoheadrightarrow E$ with $P\in\mathcal{C}$.
\end{itemize}

A regular epi $P\twoheadrightarrow E$ with $P\in\mathcal{C}$ is called a $\mathcal{C}$-\emph{cover} of $E$.

Even if $\mathcal{E}$ has limits of some type, $\mathcal{C}$ will in general only have \emph{weak} limits of that type. So if $\mathcal{E}$ has finite limits (e.g. if it is regular), then $\mathcal{C}$ will be \emph{weakly lex}, i.e. will have all weak finite limits. To construct the weak limit of a diagram in $\mathcal{C}$ one first constructs the actual limit in the ambient category $\mathcal{E}$ and then one takes a $\mathcal{C}$-cover of the latter limit.

Finally, every weakly lex category $\mathcal{C}$ appears as a projective cover inside both its regular completion $\mathcal{C}_{reg}$ and its exact completion $\mathcal{C}_{ex}$ in the sense of \cite{CV}.

\subsection{Multi-pointed categories and stars}
We first recall here some basic notions introduced in \cite{Good Ideals}.

A \emph{multi-pointed category} is a pair $(\mathcal{C},\mathcal{N})$ consisting of a category $\mathcal{C}$ and a distinguished class $\mathcal{N}$ of morphisms in $\mathcal{C}$ which is an \emph{ideal}. The latter, as mentioned in the introduction, means that for any pair of arrows $f:X\rightarrow Y$ and $g:Y\rightarrow Z$ in $\mathcal{C}$, either $f\in\mathcal{N}$ or $g\in\mathcal{N}$ implies that $gf\in\mathcal{N}$. The elements of $\mathcal{N}$ are usually referred to as \emph{null} morphisms.

We will often by abuse say that $\mathcal{C}$ is a multi-pointed category and suppress the ideal $\mathcal{N}$ if there is no possibility of confusion.

An $\mathcal{N}$-\emph{kernel} of a morphism $f:X\rightarrow Y$ is a morphism $k:K\rightarrow X$ such that $fk\in\mathcal{N}$ and which is universal with this property, i.e. whenever $fg\in\mathcal{N}$ there is a unique morphism $u$ such that $ku=g$. Note that $k$ is then necessarily a monomorphism. Observe also that this gives the usual notion of kernel in the pointed context, while in the total context the $\mathcal{N}$-kernels are just identities.

Since we shall have occasion to deal with categories that only have weak finite limits, we will also correspondingly require the notion of \emph{weak} $\mathcal{N}$-\emph{kernel} of a morphism $f:X\rightarrow Y$. This is defined as $\mathcal{N}$-kernels above, but by only requiring existence of the factorization, not necessarily uniqueness.

If $\mathcal{N}$-kernels exist for all morphisms in $\mathcal{C}$, then we shall say that $\mathcal{C}$ is a \emph{multi-pointed category with kernels}. Similarly for weak $\mathcal{N}$-kernels.

We also record here for future use the following basic observation on the behavior of $\mathcal{N}$-kernels under pullback. For the sake of completeness, we also give the easy proof.
\vspace{2mm}

\begin{lemma}\label{pullbacks of N-kernels}
	Consider the following pullback square in a multi-pointed category $(\mathcal{C},\mathcal{N})$.
	\begin{center}
		\hfil
		\xy\xymatrix{K'\ar[r]^{k'}\ar[d]_{g'} & X\ar[d]^{g} \\
			K\ar[r]_{k} & Y}\endxy
		\hfil 
	\end{center} 
	If $k$ is the $\mathcal{N}$-kernel of some $f:Y\rightarrow Z$, then $k'$ is the $\mathcal{N}$-kernel of $fg:X\rightarrow Z$.
\end{lemma}
\begin{proof}
	Let $h:A\rightarrow X$ be such that $fgh\in\mathcal{N}$. Then, since $k$ is the $\mathcal{N}$-kernel of $f$, there exists a unique $u:A\rightarrow K$ such that $ku=gh$. Now the universal property of the pullback gives a unique $v:A\rightarrow K'$ such that $g'v=u$ and $k'v=h$.
	Finally, note that $k'$ is monic because $k$ is monic.
\end{proof}
\vspace{5mm}

A pair of morphisms $r=(r_{0},r_{1}):R\rightrightarrows X$ is called a \emph{star} if $r_{0}\in\mathcal{N}$. When it is moreover jointly monic, we say that it is a \emph{star relation}.

Given a relation $R$ on an object $X$ represented by the jointly monic pair $r=(r_{0},r_{1}):R\rightrightarrows X$ and assuming $\mathcal{N}$-kernels exist, we define the \emph{star} of $R$ to be the relation $R^{*}$ on $X$ represented by the pair $(r_{0}k_{0},r_{1}k_{0})$ where $k_{0}:K_{0}\rightarrow R$ is the $\mathcal{N}$-kernel of $r_{0}$. Equivalently, one could say that $R^{*}$ is the largest subrelation of $R$ which is a star. In particular, when $R=\Eq(f)$ is the kernel pair of a morphism $f:X\rightarrow Y$, then $R^{*}=\Eq(f)^{*}$ is called the \emph{star-kernel} of $f$.

In the context of a regular multi-pointed category it is possible to use the usual calculus of relations to develop a calculus of star relations, as is done in \cite{Diamonds}. We shall not really need much of this though. We just record here the fact that, given relations $R,S$ on an object $X$, we have that $(RS)^{*}=RS^{*}$.
\vspace{1mm}

Finally, we record an observation on how the star of a relation on an object $X$ can be computed as a certain pullback involving the $\mathcal{N}$-kernel $\kappa_{X}:K_{X}\rightarrowtail X$ of the identity $1_{X}:X\rightarrow X$. Observe that this is just a generalization of the fact that the 0-class of a relation $R\rightarrowtail X\times X$ in a pointed category can be computed as the pullback of that relation along $\langle 0,1\rangle:X\rightarrow X\times X$. Since the more general statement does not appear in the literature, we also provide a proof.

\begin{lemma}\label{Star as pullback}
	Consider a relation $\xy\xymatrix{R\ar@<1ex>[r]^{r_{0}}\ar@<-1ex>[r]_{r_{1}} & X}\endxy$ in the multi-pointed category $(\mathcal{E},\mathcal{N})$ with kernels. Then the $\mathcal{N}$-kernel $k_{0}$ of $r_{0}$ is obtained as the following pullback.
	\begin{center}
		\hfil
		\xy\xymatrixcolsep{3pc}\xymatrix{K_{0}\ar@{>->}[r]^{k_{0}}\ar@{>->}[d]_{\langle \overline{r_{0}}, r_{1}k_{0}\rangle} & R\ar@{>->}[d]^{\langle r_{0},r_{1}\rangle} \\
			K_{X}\times X\ar@{>->}[r]_{\kappa_{X}\times 1_{X}} & X\times X}\endxy
		\hfil 
	\end{center}
	where $\kappa_{X}:K_{X}\rightarrowtail X$ is the $\mathcal{N}$-kernel of the identity $1_{X}:X\rightarrow X$.
\end{lemma}
\begin{proof}
	We consider the $\mathcal{N}$-kernel $k_{0}:K_{0}\rightarrowtail R$ of $r_{0}$ and we will show that there is a pullback square as indicated.
	
	First, observe that $r_{0}k_{0}\in\mathcal{N}$ implies that there is a $\overline{r_{0}}:K_{0}\rightarrow K_{X}$ such that $\kappa_{X}\overline{r_{0}}=r_{0}k_{0}$, giving the indicated morphism $K_{0}\rightarrow K_{X}\times X$ in the above commutative diagram.
	
	Now assume that $f=\langle f_{0},f_{1}\rangle:Z\rightarrow K_{X}\times X$ and $g:Z\rightarrow R$ are such that $(\kappa_{X}\times 1_{X})f=\langle r_{0},r_{1}\rangle g$. Then $r_{0}g=\kappa_{X}f_{0}$ and $r_{1}g=f_{1}$. Since $\kappa_{X}\in\mathcal{N}$, the first of these implies that $r_{0}g\in\mathcal{N}$ and hence there exists a unique $h:Z\rightarrow K_{0}$ such that $k_{0}h=g$. Then also $\langle \overline{r_{0}},r_{1}k_{0}\rangle h=\langle \overline{r_{0}}h,r_{1}k_{0}h\rangle=\langle \overline{r_{0}}h,r_{1}g\rangle=\langle f_{0},f_{1}\rangle =f$, where $\overline{r_{0}}h=f_{0}$ follows because $\kappa_{X}\overline{r_{0}}h=r_{0}k_{0}h=r_{0}g=\kappa_{X}f_{0}$ and $\kappa_{X}$ is monic.
\end{proof}

\section{2-star-permutable categories}

Let us recall the definition of 2-star-permutability from \cite{Diamonds}.

\begin{definition}
	Let $\mathcal{C}$ be a regular multi-pointed category with kernels. We say that $\mathcal{C}$ is \emph{2-star-permutable} if for any two effective equivalence relations $R,S$ on an object $X\in\mathcal{C}$ we have $RS^{*}=SR^{*}$.
\end{definition}

In the case of a pointed variety of universal algebras, the star of a relation $R$ on $X$ is the subrelation $R^{*}=\{(0,x)\in X\times X | (0,x)\in R \}$. More generally, in any pointed context, the star of the relation $\langle r_{0}, r_{1}\rangle: R\rightarrowtail X\times X$ is the relation $\langle 0,c\rangle: C\rightarrowtail X\times X$ where $c:C\rightarrowtail X$ is the \emph{0-class} of $R$, i.e. where the mono $c:C\rightarrowtail X$ is given by $c=r_{1}\ker(r_{0})$. Thus, the above definition says precisely that effective equivalence relations are \emph{0-permutable} and this is known to characterize regular subtractive categories (see \cite{Closedness 1},\cite{Diamonds} and \cite{AU} for the varietal case). In the total context, since the star of any relation is that relation itself, the definition says that effective equivalence relations are permutable, which yields precisely the regular Mal'tsev categories \cite{CKP}.
\vspace{5mm}

We first want to present an equivalent characterization of 2-star-permutability in terms of a symmetry property of reflexive relations. The symmetry property in question will be the following.

\begin{definition}\label{star-symmetry def}
	Let $\mathcal{E}$ be a multi-pointed category with kernels and $\xy\xymatrix{R\ar@<1ex>[r]^{r_{0}}\ar@<-1ex>[r]_{r_{1}} & X}\endxy$ a relation in $\mathcal{E}$. We say that $R$ is \emph{left star-symmetric} if $R^{*}\leq(R^{\circ})^{*}$. We say that it is \emph{star-symmetric} if $R^{*}=(R^{\circ})^{*}$, i.e. if both $R$ and $R^{\circ}$ are left star-symmetric.
\end{definition}

Observe that in the pointed context left star-symmetry becomes the usual notion of \emph{left 0-symmetry}, i.e. the statement that $R$ satisfies the implication $(0,x)\in_{A} R\implies (x,0)\in_{A} R$ for any generalized element $x:A\rightarrow X$ of $X$. In the total context on the other hand, $R$ being left star-symmetric just means that $R\leq R^{\circ}$, which is to say that $R$ is a symmetric relation in the ordinary sense. In particular, in this case left star-symmetry and star-symmetry become equivalent.

Indeed, note more generally that for any generalized elements $x,y:A\rightarrow X$ in $\mathcal{E}$ we have that $(x,y)\in_{A}R^{*}$ precisely if $(x,y)\in_{A}R$ and $x\in\mathcal{N}$. Thus, $R$ being left-star symmetric is saying that whenever $(n,y)\in_{A}R$ with $n\in\mathcal{N}$, then also $(y,n)\in_{A}R$.
\vspace{2mm}

We will need the following lemma, from \cite{Diamonds}, for the proof of our next proposition.

\begin{lemma}\label{inverse image of star}
	For any morphism $f:X\rightarrow Y$ and every relation $S$ on $Y$ in a multi-pointed category we have $(f^{-1}(S))^{*}=(f^{-1}(S^{*}))^{*}$.
\end{lemma}	
\vspace{3mm}

We can now present new equivalent characterizations of 2-star-permutability using the notion of star-symmetry. In fact, this allows us to also deduce that 2-star-permutability is equivalent to having the equality $RS^{*}=SR^{*}$ for any two equivalence relations $R,S$ on the same object, not just effective ones. This does not seem to have appeared in the literature before.
\vspace{2mm}

\begin{proposition}\label{2-star-permutability characterizations}
	For a regular multi-pointed category $\mathcal{C}$ with kernels the following are equivalent:
	\begin{enumerate}
		\item $\mathcal{C}$ is 2-star-permutable.
		\item For any two equivalence relations $R,S$ on an object $X\in\mathcal{C}$ we have $RS^{*}=SR^{*}$.
		\item Every reflexive relation $E$ in $\mathcal{C}$ is left star-symmetric, i.e. $E^{*}\leq(E^{\circ})^{*}$.
		\item Every reflexive relation $E$ in $\mathcal{C}$ is star-symmetric, i.e. $E^{*}=(E^{\circ})^{*}$.
	\end{enumerate}
\end{proposition}
\begin{proof}
	$1.\implies 4.$ Let $\xy\xymatrix{E\ar@<1ex>[r]^{e_{0}}\ar@<-1ex>[r]_{e_{1}} & X}\endxy$ be a reflexive relation with diagonal $\delta:X\rightarrow E$. Set $R\coloneqq \Eq(e_{0})=e_{0}^{\circ}e_{0}$ and $S\coloneqq \Eq(e_{1})=e_{1}^{\circ}e_{1}$, so that both $R$ and $S$ are effective equivalence relations on $E$. Observe that $\delta^{-1}(SR)=\delta^{\circ}e_{1}^{\circ}e_{1}e_{0}^{\circ}e_{0}\delta=e_{1}e_{0}^{\circ}=E$ and $\delta^{-1}(RS)=\delta^{\circ}e_{0}^{\circ}e_{0}e_{1}^{\circ}e_{1}\delta=e_{0}e_{1}^{\circ}=E^{\circ}$. Now using the assumption (1) and \ref{inverse image of star} we have 
	\begin{eqnarray*}
	RS^{*}=SR^{*} & \implies & (RS)^{*}=(SR)^{*} \\
	              & \implies & \delta^{-1}((RS)^{*})=\delta^{-1}((SR)^{*}) \\
	              & \implies & \delta^{-1}((RS)^{*})^{*}=\delta^{-1}((SR)^{*})^{*} \\
	              & \implies & \delta^{-1}(RS)^{*}=\delta^{-1}(SR)^{*} \\
	              & \implies & (E^{\circ})^{*}=E^{*}.
	\end{eqnarray*}

	$4.\implies 2.$ Consider the reflexive relation $E\coloneqq SR$ on $X$. Then we have $E^{*}=(E^{\circ})^{*}\implies (SR)^{*}=(RS)^{*}\implies SR^{*}=RS^{*}$.
		
	$2.\implies 1.$ Clear.
	
	$3.\iff 4.$ Clear by considering both reflexive relations $E$ and $E^{\circ}$.
\end{proof}
\vspace{4mm}

It should be observed that conditions (3) and (4) above can be formulated in any finitely complete multi-pointed category $(\mathcal{C},\mathcal{N})$ with kernels, thus enlarging the class of categories for which the notion of 2-star-permutability can be considered to include non-regular ones. This generalizes the fact that the notion of Mal'tsev category can be formulated as a finitely complete category where every reflexive relation is symmetric \cite{CKP}, as well as the fact that subtractive categories can be defined as pointed finitely complete categories where every reflexive relation is 0-symmetric \cite{Closedness 1}. The following definition therefore appears pertinent.

\begin{definition}\label{star-Mal'tsev}
	A multi-pointed category $(\mathcal{C},\mathcal{N})$ is said to be \emph{star-Mal'tsev} if every reflexive relation in $\mathcal{C}$ is left star-symmetric. Equivalently, if every reflexive relation in $\mathcal{C}$ is star-symmetric.
\end{definition}

With this terminology, \ref{2-star-permutability characterizations} says that a regular multi-pointed category is 2-star-permutable if and only if it is star-Mal'tsev.

We now want to characterize the projective covers of 2-star-permutable regular multi-pointed categories, or, in other words, of regular star-Mal'tsev categories. In doing so, the following notion will play a key role. It is the appropriate adaptation of the notion of star-symmetry to the context of a multi-pointed category with only weak finite limits and weak kernels.

\begin{definition}\label{graph star-symmetry}
	Let $(\mathcal{C},\mathcal{N})$ be a weakly lex multi-pointed category with weak $\mathcal{N}$-kernels. A graph $\xy\xymatrix{G\ar@<1ex>[r]^{g_{0}}\ar@<-1ex>[r]_{g_{1}} & X}\endxy$ in $\mathcal{C}$ is said to be \emph{left star-symmetric} if , given weak $\mathcal{N}$-kernels $k_{0}:K_{0}\rightarrow G$ and $k_{1}:K_{1}\rightarrow G$ of $g_{0}$ and $g_{1}$ respectively, there exists a $\sigma:K_{0}\rightarrow K_{1}$ such that the following diagram serially commutes
	\begin{center}
		\hfil
		\xy\xymatrix{K_{0}\ar@<1ex>[drr]^{g_{0}k_{0}}\ar@<-1ex>[drr]_{g_{1}k_{0}}\ar[rrrr]^{\sigma} &  &  &  & K_{1}\ar@<1ex>[dll]^{g_{1}k_{1}}\ar@<-1ex>[dll]_{g_{0}k_{1}}  \\
		                               &  & X&  &}\endxy
		\hfil 
	\end{center}
	i.e. such that $g_{1}k_{1}\sigma=g_{0}k_{0}$ and $g_{0}k_{1}\sigma=g_{1}k_{0}$ both hold. We say that it is \emph{star-symmetric} if both G and its opposite graph $\xy\xymatrix{G\ar@<1ex>[r]^{g_{1}}\ar@<-1ex>[r]_{g_{0}} & X}\endxy$ are left star-symmetric.
\end{definition}

In other words, a graph $G$ is left star-symmetric if a ``weak star'' of $G$ factors through a weak star of the opposite graph. Note also that the definition does not depend on the chosen weak $\mathcal{N}$-kernels because any two weak $\mathcal{N}$-kernels of the same morphism factor through each other. Furthermore, it is clear that when $G$ is a relation and $\mathcal{N}$-kernels exist the definition says precisely that $G^{*}\leq(G^{\circ})^{*}$, i.e. that $G$ is a left star-symmetric relation.

\begin{remark}
It is easy to see that in the total context we get the usual definition of a symmetric graph, since both $\mathcal{N}$-kernels are identities in this case. In the pointed context one of the two commutativities required above becomes trivial because $g_{0}k_{0}=0=g_{1}k_{1}$ and we obtain the notion of a \emph{left 0-symmetric} graph.
\end{remark}

We now introduce the categories that will appear in our characterization of the projective covers of 2-star-permutable categories. These are the multi-pointed categories with weak finite limits and weak kernels which satisfy the appropriate ``weakening'' of the star-Mal'tsev property.  Our terminology is inspired by that of Rosick\'y-Vitale in \cite{RV} for the total context.

\begin{definition}\label{star-G-Mal'tsev}
	We will say that a weakly lex multi-pointed category $\mathcal{C}$ with weak kernels is \emph{star-G-Mal'tsev} if every reflexive graph in $\mathcal{C}$ is left star-symmetric. Equivalently, if every reflexive graph is star-symmetric.
\end{definition}

\vspace{5mm}

In what follows, we will be considering regular categories $\mathcal{E}$ together with a projective cover $\mathcal{C}$ of $\mathcal{E}$. We are thus interested in how ideals of morphisms in the projective cover are related to ideals of morphisms in the ambient regular category. A thorough analysis of this situation is contained in \cite{Star Regularity}, from which we now borrow and record below the main points that will be of use in the remainder of this paper.

First, if $\mathcal{N}$ is an ideal of morphisms in the regular category $\mathcal{E}$, then we denote by $\mathcal{N}_{\mathcal{C}}$ the restriction of $\mathcal{N}$ to $\mathcal{C}$. It is then clear that $\mathcal{N}_{\mathcal{C}}$ is an ideal in $\mathcal{C}$.

Second, if we are given an ideal $\mathcal{N}$ in the projective cover $\mathcal{C}$, then we define $\mathcal{N}^{\mathcal{E}}$ to be the collection of morphisms $f:X\rightarrow Y$ in $\mathcal{E}$ for which there exists a commutative square
\begin{center}
	\hfil
	\xy\xymatrix{P\ar[r]^{n}\ar@{->>}[d]_{p} & Q\ar@{->>}[d]^{q} \\
	             X\ar[r]_{f} & Y}\endxy
	\hfil 
\end{center}
where $p$ and $q$ are regular epimorphisms and $n\in\mathcal{N}$. It is again not hard to check that $\mathcal{N}^{\mathcal{E}}$ is an ideal in $\mathcal{E}$.

\begin{lemma}\label{Ideals and proj covers}\cite{Star Regularity}
	Let $\mathcal{E}$ be a regular category having a projective cover $\mathcal{C}$.
	\begin{enumerate}
		\item For any ideal $\mathcal{N}$ in $\mathcal{E}$, if $\mathcal{E}$ has $\mathcal{N}$-kernels, then $\mathcal{C}$ has weak $\mathcal{N}_{\mathcal{C}}$-kernels, which can be computed by taking a projective cover of the domain of the $\mathcal{N}$-kernel in $\mathcal{E}$.
		\item For any ideal $\mathcal{N}$ in $\mathcal{C}$, the category $\mathcal{C}$ has weak $\mathcal{N}$-kernels if and only if $\mathcal{E}$ has $\mathcal{N}^{\mathcal{E}}$-kernels.
		\item For any ideal $\mathcal{N}$ in $\mathcal{C}$ we have $(\mathcal{N}^{\mathcal{E}})_{\mathcal{C}}=\mathcal{N}$.
		\item For any ideal $\mathcal{N}$ in $\mathcal{C}$, regular epimorphisms are $\mathcal{N}^{\mathcal{E}}$-saturating in $\mathcal{E}$ (see \ref{saturating morphism def}).
	\end{enumerate}
\end{lemma}
\vspace{5mm}

Before presenting our characterization, we find it useful to isolate the following fundamental observation. 
\vspace{0.3cm}

\begin{lemma}\label{graph-relation star-symmetry}
	Let $\mathcal{C}$ be a projective cover of the regular multi-pointed category $(\mathcal{E},\mathcal{N})$ with kernels. Consider a graph $\xy\xymatrix{G\ar@<1ex>[r]^{g_{0}}\ar@<-1ex>[r]_{g_{1}} & X}\endxy$ in $\mathcal{C}$ with its image factorization $\langle g_{0},g_{1}\rangle=\xy\xymatrix@=3em{G\ar@{->>}[r]^{q} & R\ar@{>->}[r]^(0.4){\langle r_{0},r_{1}\rangle} & X\times X }\endxy$ in $\mathcal{E}$. Then $G$ is a left star-symmetric graph if and only if $R$ is a left star-symmetric relation.
\end{lemma}
\begin{proof}
	Consider $\mathcal{N}$-kernels $k_{i}:K_{i}\rightarrow R$ of $r_{i}$, for $i=0,1$. Form the pullbacks below for $i=0,1$ and then take $\mathcal{C}$-covers $\epsilon_{i}:P_{i}\twoheadrightarrow K'_{i}$.
	\begin{center}
		\hfil
		\xy\xymatrix{K'_{i}\ar@{->>}[r]^{v_{i}}\ar@{>->}[d]_{k'_{i}} & K_{i}\ar@{>->}[d]^{k_{i}}  \\
			G\ar@{->>}[r]_{q} & R}\endxy
		\hfil 
	\end{center}
	By \ref{pullbacks of N-kernels} we know that $k'_{i}:K'_{i}\rightarrowtail G$ is the $\mathcal{N}$-kernel of $r_{i}q=g_{i}$. Thus, we have that $u_{i}\coloneqq k'_{i}\epsilon_{i}:P_{i}\rightarrow G$ is a weak $\mathcal{N}$-kernel of $g_{i}$ in $\mathcal{C}$ for $i=0,1$ by \ref{Ideals and proj covers}.
	
	Assume first that $R$ is left star-symmetric, so that there exists a morphism $\sigma:K_{0}\rightarrow K_{1}$ such that $r_{1}k_{1}\sigma=r_{0}k_{0}$ and $r_{0}k_{1}\sigma=r_{1}k_{0}$. By projectivity of $P_{0}$ and the fact that $v_{1}\epsilon_{1}$ is a regular epi, there exists a morphism $\tilde{\sigma}:P_{0}\rightarrow P_{1}$ making the following diagram commute.
	\begin{center}
		\hfil
		\xy\xymatrixcolsep{3pc}\xymatrix{P_{0}\ar@{->>}[r]^{v_{0}\epsilon_{0}}\ar@{-->}[dd]_{\tilde{\sigma}} & K_{0}\ar@<1ex>[drr]^{r_{0}k_{0}}\ar@<-1ex>[drr]_{r_{1}k_{0}}\ar[dd]^{\sigma} & &  \\
			&                &  &     X\\
			P_{1}\ar@{->>}[r]_{v_{1}\epsilon_{1}} & K_{1}\ar@<1ex>[urr]^{r_{1}k_{1}}\ar@<-1ex>[urr]_{r_{0}k_{1}} &  &}\endxy
		\hfil 
	\end{center}
	
	Now we have
	\begin{eqnarray*}
		g_{1}u_{1}\tilde{\sigma} & = & r_{1}qk'_{1}\epsilon_{1}\tilde{\sigma} \\
		                         & = & r_{1}k_{1}v_{1}\epsilon_{1}\tilde{\sigma} \\
		                         & = & r_{1}k_{1}\sigma v_{0}\epsilon_{0} \\
		                         & = & r_{0}k_{0}v_{0}\epsilon_{0} \\
		                         & = & r_{0}qk'_{0}\epsilon_{0} \\
		                         & = & g_{0}u_{0}
	\end{eqnarray*}	
	
	and similarly
		
	\begin{eqnarray*}
		g_{0}u_{1}\tilde{\sigma} & = & r_{0}qk'_{1}\epsilon_{1}\tilde{\sigma} \\
		& = & r_{0}k_{1}v_{1}\epsilon_{1}\tilde{\sigma} \\
		& = & r_{0}k_{1}\sigma v_{0}\epsilon_{0} \\
		& = & r_{1}k_{0}v_{0}\epsilon_{0} \\
		& = & r_{1}qk'_{0}\epsilon_{0} \\
		& = & g_{1}u_{0}
	\end{eqnarray*}	
	proving that $G$ is left star-symmetric.
	\vspace{1cm}
	
	Conversely, assume that $G$ is left star-symmetric. This means that there exists a $\sigma:P_{0}\rightarrow P_{1}$ such that $g_{1}u_{1}\sigma=g_{0}u_{0}$ and $g_{0}u_{1}\sigma=g_{1}u_{0}$. We can then again calculate as follows:
	\begin{eqnarray*}
		r_{1}k_{1}v_{1}\epsilon_{1}\sigma & = & r_{1}qk'_{1}\epsilon_{1}\sigma \\
		                                  & = & g_{1}k'_{1}\epsilon_{1}\sigma \\
		                                  & = & g_{1}u_{1}\sigma \\
		                                  & = & g_{0}u_{0} \\
		                                  & = & r_{0}qk'_{0}\epsilon_{0} \\
		                                  & = & r_{0}k_{0}v_{0}\epsilon_{0}
	\end{eqnarray*}

	\begin{eqnarray*}
		r_{0}k_{1}v_{1}\epsilon_{1}\sigma & = & r_{0}qk'_{1}\epsilon_{1}\sigma \\
										  & = & g_{0}k'_{1}\epsilon_{1}\sigma \\
										  & = & g_{0}u_{1}\sigma \\
										  & = & g_{1}u_{0} \\
										  & = & r_{1}qk'_{0}\epsilon_{0} \\
										  & = & r_{1}k_{0}v_{0}\epsilon_{0}
	\end{eqnarray*}

	This means that the square below commutes and so we obtain the indicated morphism $\tilde{\sigma}$ because $v_{0}\epsilon_{0}$ is a regular epi and $\langle r_{1}k_{1},r_{0}k_{1}\rangle$ is monic, being the star of the relation $R^{\circ}$.
	\begin{center}
		\hfil
		\xy\xymatrixcolsep{5pc}\xymatrix@R=3em{P_{0}\ar@{->>}[r]^{v_{0}\epsilon_{0}}\ar[d]_{v_{1}\epsilon_{1}\sigma} & K_{0}\ar@{>->}[d]^{\langle r_{0}k_{0},r_{1}k_{0}\rangle}\ar@{-->}[dl]^{\tilde{\sigma}} \\
			K_{1}\ar@{>->}[r]_{\langle r_{1}k_{1},r_{0}k_{1}\rangle} & X\times X}\endxy	
		\hfil 
	\end{center}
	The commutation of the bottom triangle is precisely left star-symmetry of $R$.
\end{proof}
\vspace{0.5cm}

In order to prove our main result, we will need to impose an additional condition on the regular category $\mathcal{E}$ regarding the behavior of regular epimorphisms with respect to $\mathcal{N}$-kernels. This condition is familiar from the literature (see \cite{Diamonds, On 2-star-perm, Star Regularity}) and is indeed mild enough that it includes all examples of interest. We now proceed to introduce the necessary notions.

Consider any object $X$ in the multi-pointed category $(\mathcal{E},\mathcal{N})$. We will denote by $\kappa_{X}:K_{X}\rightarrowtail X$ the $\mathcal{N}$-kernel of the identity morphism $1_{X}$. Observe that by definition the generalized elements of $K_{X}$ correspond precisely to the generalized elements of $X$ that are in $\mathcal{N}$. Hence, $K_{X}$ should be thought of as consisting of the ``trivial elements'' of the object $X$.

Now given any morphism $f:X\rightarrow Y$ in $\mathcal{E}$, we have a uniquely induced morphism $\tilde{f}:K_{X}\rightarrow K_{Y}$ making the following square commute.
\begin{center}
	\hfil
	\xy\xymatrix{K_{X}\ar@{-->}[d]_{\tilde{f}}\ar@{>->}[r]^{\kappa_{X}} & X\ar[d]^{f} \\
		K_{Y}\ar@{>->}[r]_{\kappa_{Y}} & Y}\endxy
	\hfil 
\end{center}

Then we can introduce the following definition.

\begin{definition}\label{saturating morphism def}
	A morphism $f:X\rightarrow Y$ in a multi-pointed category $(\mathcal{E},\mathcal{N})$ is called \emph{saturating} if the induced morphism $\tilde{f}:K_{X}\rightarrow K_{Y}$ is a regular epimorphism.
\end{definition}

Note that in the pointed context all morphisms are saturating, since $K_{X}=0$ for any object $X$. In the total context, on the other hand, it is clear that the saturating morphisms are precisely the regular epimorphisms. In fact, that all regular epimorphisms are saturating is precisely what we shall require below. 
\vspace{3mm}

Now we can present the main result of this note.
\vspace{0.3cm}

\begin{theorem}\label{main result}
	Let $\mathcal{C}$ be a projective cover of the regular multi-pointed category with kernels $(\mathcal{E},\mathcal{N})$ and assume regular epimorphisms in $\mathcal{E}$ are saturating. Then $(\mathcal{E},\mathcal{N})$ is 2-star-permutable if and only if $(\mathcal{C},\mathcal{N}_{\mathcal{C}})$ is star-G-Mal'tsev.
\end{theorem}
\begin{proof}
	Assume first that $\mathcal{E}$ is 2-star-permutable and consider any reflexive graph $\xy\xymatrix{G\ar@<1ex>[r]^{g_{0}}\ar@<-1ex>[r]_{g_{1}} & X}\endxy$ in $\mathcal{C}$ with splitting $\delta:X\rightarrow G$. Consider its image factorization $\langle g_{0},g_{1}\rangle=\xy\xymatrix@=3em{G\ar@{->>}[r]^{q} & R\ar@{>->}[r]^(0.4){\langle r_{0},r_{1}\rangle} & X\times X }\endxy$ in the regular category $\mathcal{E}$. Then the relation $R$ on $X$ is reflexive as well, since $r_{i}q\delta=g_{i}\delta=1_{X}$ for $i=0,1$. Since $\mathcal{E}$ is 2-star-permutable, we know by \ref{2-star-permutability characterizations} that $R$ must be star-symmetric. Now \ref{graph-relation star-symmetry} implies that $G$ is (left) star-symmetric.
	\vspace{1cm}

	Conversely, assume that $(\mathcal{C},\mathcal{N}_{\mathcal{C}})$ is star-G-Mal'tsev. Consider any reflexive relation $\xy\xymatrix{E\ar@<1ex>[r]^{e_{0}}\ar@<-1ex>[r]_{e_{1}} & X}\endxy$ in $\mathcal{E}$. We want to show that $E$ is left star-symmetric.
	
	Take a $\mathcal{C}$-cover $p:\tilde{X}\twoheadrightarrow X$ of $X$ and consider the inverse image relation $E'\coloneqq p^{-1}(E)$. i.e. form the following pullback
	\begin{center}
		\hfil
		\xy\xymatrix@=3em{E'\ar@{>->}[r]^(0.4){\langle e'_{0},e'_{1}\rangle}\ar@{->>}[d]_{q} & \tilde{X}\times\tilde{X}\ar@{->>}[d]^{p\times p} \\
			E\ar@{>->}[r]^(0.4){\langle e_{0},e_{1}\rangle} & X\times X}\endxy
		\hfil 
	\end{center}
	Now again take a $\mathcal{C}$-cover $\epsilon:G\twoheadrightarrow E'$ and set $g_{0}\coloneqq e'_{0}\epsilon$ and $g_{1}\coloneqq e'_{1}\epsilon$, so that we have a graph $\xy\xymatrix{G\ar@<1ex>[r]^{g_{0}}\ar@<-1ex>[r]_{g_{1}} & \tilde{X}}\endxy$ in $\mathcal{C}$. Observe that the relation $E'$ is reflexive, being the inverse image of a reflexive relation. It follows that the graph $G$ is also reflexive. Indeed, if $\delta':\tilde{X}\rightarrow E'$ is the diagonal of $E'$, then by projectivity of $G$ we can lift to a $\tilde{\delta}:\tilde{X}\rightarrow G$ such that $\epsilon\tilde{\delta}=\delta'$ and then $g_{i}\tilde{\delta}=e'_{i}\epsilon\tilde{\delta}=e'_{i}\delta'=1_{\tilde{X}}$ for $i=0,1$.
	
	Now consider $\mathcal{N}$-kernels $k_{i}:K_{i}\rightarrow E$ of $e_{i}$ and $k'_{i}:K'_{i}\rightarrow E$ of $e'_{i}$ in $\mathcal{E}$ for $i=0,1$. We then have induced morphisms $u_{i}:K'_{i}\rightarrow K_{i}$ such that $k_{i}u_{i}=qk'_{i}$. We claim that the $u_{i}$ are regular epimorphisms.
	\begin{center}
		\hfil
		\xy\xymatrix{K'_{i}\ar@{>->}[r]^{k'_{i}}\ar@{->>}[d]_{u_{i}} & E'\ar@{->>}[d]^{q} \\
			K_{i}\ar@{>->}[r]_{k_{i}} & E}\endxy
		\hfil 
	\end{center}
	
	To see this for $u_{0}$ we consider the following two commutative diagrams. In the first one the right-hand square is a pullback by construction, while the left-hand square is a pullback by \ref{Star as pullback}. In the second diagram we know only that the right-hand square is a pullback, again by \ref{Star as pullback}.
	\begin{center}
		\hfil
		\xy\xymatrixcolsep{5pc}\xymatrix{K'_{0}\ar@{>->}[r]^{k'_{0}}\ar[d]_{\langle \overline{e'_{0}}, e'_{1}k'_{0}\rangle} & E'\ar@{->>}[r]^{q}\ar@{>->}[d]_{\langle e'_{0},e'_{1}\rangle} & E\ar@{>->}[d]^{\langle e_{0},e_{1}\rangle} \\
		             K_{\tilde{X}}\times \tilde{X}\ar[r]_{\kappa_{\tilde{X}\times 1_{\tilde{X}}}} & \tilde{X}\times\tilde{X}\ar@{->>}[r]_{p\times p} & X\times X}\endxy
		\hfil 
	\end{center}

	\begin{center}
		\hfil
		\xy\xymatrixcolsep{5pc}\xymatrix{K'_{0}\ar@{->>}[r]^{u_{0}}\ar[d]_{\langle \overline{e'_{0}}, e'_{1}k'_{0}\rangle} & K_{0}\ar@{>->}[r]^{k_{0}}\ar[d]_{\langle \overline{e_{0}}, e_{1}k_{0}\rangle} & E\ar@{>->}[d]^{\langle e_{0},e_{1}\rangle} \\
		             K_{\tilde{X}}\times \tilde{X}\ar@{->>}[r]_{\tilde{p}\times p} & K_{X}\times X\ar[r]_{\kappa_{X}\times 1_{X}} & X\times X}\endxy
		\hfil 
	\end{center}

	Since $(p\times p)(\kappa_{\tilde{X}}\times1_{\tilde{X}})=(\kappa_{X}\times 1_{X})(\tilde{p}\times p)$, we deduce that the outer rectangle in the second diagram is a pullback. Then by the usual pullback-cancellation property we have that the left-hand square is a pullback as well. But since both $p$ and $\tilde{p}$ are regular epis, so is $\tilde{p}\times p$, since $\mathcal{E}$ is regular, and hence we deduce that the pullback $u_{0}$ is a regular epi.
	
	By the assumption that $\mathcal{C}$ is star-G-Mal'tsev, the reflexive graph $G$ is left star-symmetric and so by \ref{graph-relation star-symmetry} its image relation $E'$ is also left star-symmetric. Thus, there exists a $\sigma':K'_{0}\rightarrow K'_{1}$ such that $e_{1}'k_{1}'\sigma'=e_{0}'k_{0}'$ and $e'_{0}k'_{1}\sigma'=e'_{1}k'_{0}$.
	
	Finally, consider the commutative square below.
	\begin{center}
		\hfil
		\xy\xymatrixcolsep{5pc}\xymatrix@R=3em{K'_{0}\ar@{->>}[r]^{u_{0}}\ar[d]_{u_{1}\sigma'} & K_{0}\ar[d]^{\langle e_{0}k_{0},e_{1}k_{0}\rangle}\ar@{-->}[dl]^{\sigma} \\
			K_{1}\ar@{>->}[r]_(0.5){\langle e_{1}k_{1},e_{0}k_{1}\rangle} & X\times X}\endxy
		\hfil 
	\end{center}
	Since $u_{0}$ is a regular epi and $\langle e_{1}k_{1},e_{0}k_{1}\rangle$ is monic (being the star of the relation $E^{\circ}$), we get the indicated factorization $\sigma:K_{0}\rightarrow K_{1}$, which shows that $E$ is left star-symmetric. Thus, by \ref{2-star-permutability characterizations} it follows that $\mathcal{E}$ is 2-star-permutable.
\end{proof}
\vspace{3mm}

The above result yields a characterization of when the regular and exact completion (in the sense of \cite{CV}) of a weakly lex multi-pointed category are 2-star-permutable.

\begin{corollary}
	Let $(\mathcal{C},\mathcal{N})$ be a weakly lex multi-pointed category with weak kernels. Then $(\mathcal{C}_{reg}, \mathcal{N}^{\mathcal{C}_{reg}})$ is 2-star-permutable if and only if $(\mathcal{C}, \mathcal{N})$ is star-G-Mal'tsev.
\end{corollary}
\begin{proof}
	$\mathcal{C}$ appears as a projective cover inside $\mathcal{C}_{reg}$. Then \ref{main result} indeed applies to give the result because by \ref{Ideals and proj covers} we know that $\mathcal{C}_{reg}$ has $\mathcal{N}^{\mathcal{C}_{reg}}$-kernels, regular epimorphisms in $\mathcal{C}_{reg}$ are $\mathcal{N}^{\mathcal{C}_{reg}}$-saturating and $(\mathcal{N}^{\mathcal{C}_{reg}})_{\mathcal{C}}=\mathcal{N}$.
\end{proof}

In the exact same way we get the corresponding result about the exact completion $\mathcal{C}_{ex}$.

\begin{corollary}
	Let $(\mathcal{C},\mathcal{N})$ be a weakly lex multi-pointed category with weak kernels. Then $(\mathcal{C}_{ex}, \mathcal{N}^{\mathcal{C}_{ex}})$ is 2-star-permutable if and only if $(\mathcal{C}, \mathcal{N})$ is star-G-Mal'tsev.
\end{corollary}
\vspace{5mm}

\begin{remark}
	We should comment here on how \ref{main result} extends the characterizations of projective covers for regular Mal'tsev categories, due to Rosicky-Vitale\cite{RV}, and for regular subtractive categories, due to Gran-Rodelo\cite{Jonnson-Tarski}.
	
	For the Mal'tsev case, it is manifestly clear that we obtain exactly the same characterization as in \cite{RV}, i.e. our star-G-Mal'tsev categories are exactly the G-Mal'tsev ones introduced therein. Note that in that paper G-Mal'tsev is initially defined by requiring that every reflexive graph be both symmetric and transitive, but this is equivalent to just requiring symmetry and that is in fact implicitly proved in \cite{RV}.
	
	In the pointed context, it is not immediate from the definitions that our star-G-Mal'tsev, which we should probably call \emph{0-G-Mal'tsev} in this case, yields the \emph{w-subtractive} categories of Gran-Rodelo\cite{Jonnson-Tarski}. Of course, since both characterize projective covers of the same class of regular categories, they turn out to be equivalent, since any weakly lex category can always be considered a projective cover of its regular completion. On the other hand, a direct proof of the equivalence of the two notions is also not too hard to construct.
\end{remark}
\vspace{3mm}

We now recall a particular class of multi-pointed categories, the so-called \emph{proto-pointed context} introduced in \cite{Good Ideals}. This refers to a regular category in which every object has a smallest subobject and where a morphism $f:X\rightarrow Y$ is a null morphism precisely when it factors through the smallest subobject of its codomain $Y$. In the case of a variety $\mathbb{V}$ of universal algebras these morphisms are exactly those whose image is the subalgebra $E$ generated by the constants. This latter situation has been called the \emph{algebraic proto-pointed context} in \cite{Diamonds}. In this context the $\mathcal{N}$-kernel of a morphism $f:X\rightarrow Y$ is given by the subalgebra of elements $x\in X$ that are mapped by $f$ to an element in $E$. Hence, the star of a relation $R\rightarrowtail X\times X$ consists of all pairs $(e,x)\in R$ such that $e\in E$ and the relation is star-symmetric precisely if, for every $e\in E$ and $x\in X$, we have $(e,x)\in R\iff(x,e)\in R$. Observe also that, as long as the set of constants is nonempty, all morphisms will be saturating, since the condition for a homomorphism here to be saturating says precisely that it is surjective on constants. In fact, it is also not hard to see that regular epimorphisms are saturating in any proto-pointed context, not just the varietal one.

Now suppose we are in an algebraic proto-pointed context and that the set of constants of the variety is nonempty. It was proved in \cite{Diamonds} that in this case 2-star-permutability is equivalent to a priori more general properties such as \emph{3-star-permutability} and the \emph{symmetric saturation property}, but also to the syntactic condition defining \emph{E-subtractive varieties} in the sense of \cite{U}. We would like to conclude this note by showing that the equivalence with the latter notion can also directly be obtained from our results.

\begin{corollary}
	Let $\mathbb{V}$ be a variety of universal algebras with set of constants $E\neq\emptyset$. Then $\mathbb{V}$ is 2-star-permutable if and only if the following syntactic condition holds: 
	
	For every $e\in E$ there exists a binary term $s_{e}(x,y)$ such that $s_{e}(x,x)=e$ and $s_{e}(x,e)=x$.
\end{corollary}
\begin{proof}
	Suppose 2-star-permutability holds and fix any $e\in E$. We then consider a graph $\xy\xymatrix{F(x,y)\ar@<1ex>[r]^{g_{0}}\ar@<-1ex>[r]_{g_{1}} & F(x)}\endxy$ between free algebras on 2 and 1 generator respectively, where $g_{0},g_{1}$ are defined by setting $g_{0}(x)=x$, $g_{0}(y)=e$ and $g_{1}(x)=g_{1}(y)=x$. This graph is reflexive, since it is clearly split by the map $\delta:F(x)\rightarrow F(x,y)$ defined by $\delta(x)=x$. Since free algebras are projective, we can apply \ref{main result} (i.e. $\mathcal{C}$ here is the full subcategory of free algebras) to deduce that this graph must be star-symmetric.
	
	Now we have $(e,x)=(g_{0}(y),g_{1}(y))$, so by the star-symmetry we must also have $(x,e)=(g_{0}(s_{e}),g_{1}(s_{e}))$ for some $s_{e}(x,y)\in F(x,y)$. Thus, $x=g_{0}(s_{e}(x,y))=s_{e}(x,e)$ and $e=g_{1}(s_{e}(x,y))=s_{e}(x,x)$.
	\vspace{2mm}
	
	Conversely, suppose we have binary terms $s_{e}(x,y)$ for all $e\in E$ with the indicated properties. We will show that any reflexive relation $R\rightarrowtail X\times X$ in the variety $\mathbb{V}$ is left star-symmetric.
	
	Indeed, assume that $(e,x)\in R$ for some $e\in E$ and $x\in X$. Since $R$ is reflexive, we also have $(x,x)\in R$. By compatibility with the operations we then must have $(s_{e}(x,e),s_{e}(x,x))\in R$, i.e. that $(x,e)\in R$. This concludes the proof.
\end{proof}	
\vspace{2mm}

As particular examples of $E$-subtractive varieties one has the categories \textbf{Ring} of unitary rings, as well as the categories \textbf{Heyt}, \textbf{Bool} of Heyting and Boolean algebras respectively.

\begin{remark}
	Since our main result is stated for any regular category, without requiring exactness, it can equally well be applied to \emph{quasi-varieties} of universal algebras, since these are still regular categories. This encompasses further interesting examples, such as the category \textbf{RedRng} of \emph{reduced rings}, i.e. unitary rings $R$ satisfying $(\forall x\in R)(\forall n\geq 1) (x^{n}=0\implies x=0)$.
\end{remark}

\end{document}